\documentclass{article}

\usepackage{amssymb,amsmath,mathrsfs,amsthm}
\usepackage{natbib}
\usepackage{mathtools}
\usepackage{geometry}
\usepackage{amsfonts}
\usepackage{esint,comment}
 \newtheorem{theorem}{Theorem}[section]
 \newtheorem{corollary}[theorem]{Corollary}

 \theoremstyle{definition}
 \newtheorem{definition}[theorem]{Definition}
 \theoremstyle{remark}

 \numberwithin{equation}{section}

\def \no#1#2#3 {{\bf #1} (#3), #2.}
\def \eds#1#2#3 {#1, #2, #3.}

\def\e{{\rm e}}
\def\d{{\rm d}}

\def\Y{{\rm Y}}

\def\T{{\mathbb T}}

\def\:{{\colon}}

\def\be#1{\begin{equation}\label{#1}}
\def\ee{\end{equation}}

\def\<{\langle}
\def\>{\rangle}
\def\coloneqq{:=}

\newcommand{\p}{\partial}

\newcommand{\RR}{\mathbb{R}}

\newcommand{\eqnb}{\begin{equation}}
\newcommand{\eqne}{\end{equation}}

\newcommand\blfootnote[1]{%
  \begingroup
  \renewcommand\thefootnote{}\footnote{#1}%
  \addtocounter{footnote}{-1}%
  \endgroup
}

\begin{document}
\title{A sufficient integral condition for local regularity of solutions to the surface growth model}
\author{Wojciech S. O\.za\'nski}
\maketitle
\blfootnote{\noindent Mathematics Institute, Zeeman Building, University of Warwick, Coventry CV4 7AL, UK\\ w.s.ozanski@warwick.ac.uk}


\begin{abstract}
The surface growth model, $u_t + u_{xxxx} + \p_{xx} u_x^2 =0$, is a one-dimensional fourth order equation, which shares a number of striking similarities with the three-dimensional incompressible Navier--Stokes equations, including the results regarding existence and uniqueness of solutions and the partial regularity theory. Here we show that a weak solution of this equation is smooth on a space-time cylinder $Q$ if the Serrin condition $u_x\in L^{q'}L^q (Q)$ is satisfied, where $q,q'\in [1,\infty ]$ are such that either $1/q+4/q'<1$ or $1/q+4/q'=1$, $q'<\infty$.

\end{abstract}


\section{Introduction}
We consider the one-dimensional model of surface growth
\begin{equation}\label{SGM}
u_t+u_{xxxx}+\partial_{xx}u_x^2=0
\end{equation}
on the one-dimensional torus $\T$, under the assumption that $\int_{\T}u=0$. This equation originates from a model of epitaxy, see \cite{Si-Pl:94b}, \cite{MW:11}, \cite{SteWin:05}, \cite{Cu-Va-Ga:05} and \cite{Ra-Li-Ha:00a}
As observed by \cite{blomkerromito09,blomkerromito12} this model shares many striking similarities with the three-dimensional incompressible Navier--Stokes equations (NSE), including the results regarding existence and uniqueness of solutions. 
For example their 2009 paper proves local existence in the critical space $\dot{H}^{1/2}$ and (spatial) smoothness for solutions bounded in $L^{8/(2\alpha -1)} ((0,T); H^\alpha )$ for all $\alpha \in (1/2,9/2)$. The 2012 paper proves local existence in a critical space of a similar type to that occurring in the paper by \cite{koch_tataru} for the NSE.
Very recently \cite{Ozanski_Robinson_2017} developed a partial regularity theory for the surface growth model, an analogue of the celebrated \cite{CKN} theorem for the NSE. The central result of this theory is that there exists $\varepsilon_0 >0$ such any ``suitable'' weak solution of the surface growth model (see Definition \ref{def_of_weak_sol}) satisfying
\eqnb\label{local_reg_cond_SGM}
\frac{1}{r^2} \int_{Q(z,r)} |u_x|^3 < \varepsilon_0
\eqne
is H\"older continuous (with any exponent $\alpha \in (0,1)$) in $Q(z,r/2)$ (see Theorem 4.5 in \cite{Ozanski_Robinson_2017}). Here $Q(z,r)$ denotes a space-time cylinder of radius $r$ centred at $z$ and ``suitable'' refers to weak solutions satisfying a local form of an energy inequality, see \cite{Ozanski_Robinson_2017} for precise definitions. 
A straightforward consequence of this result is that one can estimate the dimension of the singular set
\eqnb\label{singset}\begin{split}
S &\coloneqq \{(x,t)\in\T\times[0,\infty):\ u\mbox{ is not space-time H\"older continuous} \\
&\hspace{4.5cm}\qquad\qquad\qquad\mbox{in any neighbourhood of }(x,t)\}.
\end{split}\eqne
Namely one obtains $d_B(S)\leq 7/6$, where $d_B$ denotes the box-counting dimension. Another consequence of the partial regularity theory is that $d_H(S)\leq 1$, where $d_H$ denotes the Hausdorff dimension. 

Unfortunately it is not known whether the definition \eqref{singset} of the singular set is optimal; that is whether $u$ is smooth (locally) outside $S$.

Interestingly, it seems that the proof of the partial regularity theory cannot be obtained by the method of \cite{CKN}; instead \cite{Ozanski_Robinson_2017} used the approach of \cite{lin} and \cite{ladyzhenskaya_seregin}; the analysis there also required a novel tool, namely a nonlinear parabolic Poincar\'e inequality,
\eqnb\label{PPoinc}
\frac{1}{r^5}\int_{Q(z,r/2)}|u-u_{z,r/2}|^3\le C_{pp} \left( \Y (z,r) + \Y (z,r)^2 \right),
\eqne
for weak solutions of the surface growth model, where
\[
\Y (z,r) \coloneqq \frac{1}{r^2} \int_{Q(z,r)} |u_x|^3
\]
and $u_{z,r/2}$ denotes the mean of $u$ over $Q(z,r/2)$. 

In this article we study another property of weak solutions of the surface growth model, which is similar to the so-called local Serrin condition in the case of the Navier--Stokes equations. In the case of the NSE this condition reads: if $u$ is a weak solution on a space-time domain $U\times (t_1,t_2)$ such that
\eqnb\label{serrin_cond_NSE}
u \in L^{q'} ((t_1,t_2); L^q (U))\quad \text{ with } \frac{3}{q}+\frac{2}{q'} \leq 1
\eqne
then $u$ is smooth in the space variables on this domain. The condition is named after Serrin (1962 \& 1963)\nocite{serrin_62,serrin_63}, who was the first to study the property \eqref{serrin_cond_NSE} in the subcritical case (that is when the inequality in \eqref{serrin_cond_NSE} is sharp ``$<$''). The critical case (that is when $2/q'+3/q=1$) has been studied by \cite{fabes_jones_riviere_72} when $q\in (3,\infty)$, and \cite{struwe_88}, \cite{takahashi_90} when $q >3$. The most difficult case of $q'=\infty, q=3$ was resolved by Escauriaza, Seregin \& \v{S}ver\'{a}k (2003)\nocite{ESS_2003}. We refer the reader to an excellent presentation of the local Serrin condition in the NSE and further references in Chapter 13 of \cite{NSE_book}.

Here, we prove that a weak solution $u$ to the surface growth model on a space-time domain $U\times (t_1,t_2)$ is smooth on this domain provided that
\eqnb\label{serrin_cond_SGM}
u_x \in L^{q'} ((t_1,t_2); L^q (U))\quad \text{ where }\quad \frac{1}{q}+\frac{4}{q'} \leq 1
\eqne
(and $q,q'\in (1, \infty ]$), see Theorem \ref{thm_subcritical_case} for the subcritical case (i.e. $1/q+4/q'<1$) and Theorem \ref{thm_critical_case} for the critical case (i.e. $1/q+4/q'=1$). Remarkably, under this condition we obtain smoothness in both space and time (rather than smoothness in space only, as in the case of the NSE). Note that from all possible choices of the exponents $q,q'$ satisfying the relation in \eqref{serrin_cond_SGM} we only exclude the choice $q=1$, $q'=\infty$ (see the conclusion for a discussion regarding this case). Moreover observe that we state the Serrin condition \eqref{serrin_cond_SGM} in terms of $u_x$ (rather than in terms of $u$, which is the case in the NSE). This is related to the fact that the lowest spatial derivative of $u$ involved in the nonlinear term in the surface growth model \eqref{SGM} is $u_x$ (rather than the $0$-th derivative, which is the case in the NSE). Note that the condition for local regularity from \cite{Ozanski_Robinson_2017} is also stated in terms of $u_x$ (see \eqref{local_reg_cond_SGM}). Therefore, our main result is a next step towards understanding the remarkable similarity between the SGM and the NSE.

Interestingly, local H\"older continuity of $u$ follows trivially from the subcritical Serrin condition by a use of an extended version of the parabolic Poincar\'e inequality. Indeed the parabolic Poincar\'e inequality \eqref{PPoinc} can be extended to
\eqnb\label{PPoinc_generalised}
\frac{1}{r^5}\int_{Q(z_0,r/2)}|u-u_{z_0,r/2}|^p\leq C \left( (r^\varepsilon M )^p + (r^\varepsilon M )^{2p} \right),
\eqne
where
\[
M\coloneqq \left[ \int_{t_1}^{t_2} \left( \int_U |u_x(t)|^p \right)^{p'/p} \d t \right]^{1/p'}
\]
and $p, p'\in [2,\infty )$ are such that $1/p+4/p'=1-\varepsilon$. Such an extension can be proved in the same way as \eqref{PPoinc} (see Theorem 3.1 in \cite{Ozanski_Robinson_2017}). Thus if the condition \eqref{serrin_cond_SGM} is satisfied with $1/q+4/q'=1-\varepsilon$, then an application of H\"older's inequality to \eqref{PPoinc_generalised} together with the Campanato lemma give $\varepsilon$-H\"older continuity of $u$ in $U\times (t_1,t_2)$.

What is more, in the same circumstances (and also allowing the case $1/q+4/q' =1$, $q'<\infty$) the partial regularity theory gives $\alpha$-H\"older continuity for any $\alpha \in (0,1)$ if $q \geq 3$. Indeed, H\"older's inequality gives
\[
\frac{1}{r^2} \int_{Q(z,r)} |u_x|^3 \leq \| u_x \|^3_{L^{q',q} (Q(z,r)) } r^{3(1-1/q-4/q')},
\]
which is less than $\varepsilon_0$ if the cylinder $Q(z,r)$ is small enough (provided $q'<\infty$, which can be guaranteed in any case, cf. the beginning of the proof of Theorem \ref{thm_subcritical_case}). Thus the condition \eqref{local_reg_cond_SGM} for partial regularity can be guaranteed for each sufficiently small cylinder $Q(z,r)$, and the claim follows. Therefore the main result of this article is interesting, since it shows that the local Serrin condition \eqref{serrin_cond_SGM} gives $C^\infty$ smoothness, rather than merely $\alpha$-H\"older continuity for any $\alpha \in (0,1)$. This also suggests that the definition of the singular set \eqref{singset} is optimal.
 
The proof of our result begins with the approach of \cite{takahashi_90}, adapted to the one-dimensional fourth order setting. However, in order to show boundedness of higher derivatives in space, we use fractional order Sobolev spaces to obtain, roughly speaking, a half of a derivative at a time (see step 2' of the proof of Theorem \ref{thm_subcritical_case}). This seems to be a novel approach in this context. Moreover, our setting requires a uniqueness theorem for weak solutions of the biharmonic heat equation $w_t+w_{xxxx}=0$. Since we are dealing with a fourth order equation, we cannot adapt any uniqueness theorem for the heat equation which uses the maximum principle. We can, however, adapt a uniqueness theorem for the heat equation which is based on the density of the image of the heat operator $\p_t- \Delta$ in $L^p$. A uniqueness theorem of this kind is presented in Section 4.4.2 of \cite{giga_saal} and we prove an appropriate adaptation of it in Theorem \ref{thm_uniqueness_sol_biharm_heat} below. 

The structure of the article is the following. In Section \ref{sec_prelims} we discuss the preliminary concepts including fractional Sobolev spaces (Section \ref{sec_fractional_sobolev}), the biharmonic heat kernel (Section \ref{sec_biharm_heat_kernel}), the uniqueness theorem (Section \ref{sec_uniqueness_thm1}) and the concept of local weak solutions to the surface growth model (Section \ref{sec_weak_sols_to_SGM}). We then proceed to the proof of the local Serrin condition for the surface growth model in Section \ref{sec_local_serrin_cond}. There, we discuss a certain representation formula for $u_x$ and we prove regularity, first for the subcritical case $1/q+4/q'<1$ in Section \ref{sec_subcritical} and then for the critical case $1/q+4/q'=1$ with $q'<\infty$ in Section \ref{sec_critical}.

\section{Preliminaries}\label{sec_prelims}
We will write $\partial_t$ for the derivative in time and $\partial_x$ for the derivative in $x$. 
We denote by $v_\varepsilon$ the mollification of a function $v:\RR^2 \to \RR $ in both space in time. We say that a function satisfies a partial differential equation in an open set if it satisfies the distributional form of the equation. We denote a space-time cylinder by $Q$, that is $Q=B\times I$ for some intervals $B,I\subset \RR$. We use the shorthand notation $\| \cdot \|_p \coloneqq   \|\cdot \|_{L^p (\RR )}$. Given $Q$ and $p,p'\in [1,\infty ]$ we let
\[
L^{p',p} (Q) = L^{p'} (I; L^p(B)) \coloneqq \{ f\in Q\to \RR \, : \, f \text{ is measurable and } \| f \|_{L^{p',p}(Q) }< \infty \},
\]
where
\[
\|f \|_{L^{p',p} (Q )}^{p'} \coloneqq \int_I  \| f(t) \|_{L^p (B)}^{p'} \d t 
\]
for $p'<\infty$ and 
\[
\|f \|_{L^{\infty,p} (Q )} \coloneqq \mathrm{esssup}_{t\in I}  \| f(t) \|_{L^p (B)}. 
\]
Given $T>0$ we use the shorthand notation $L^{p',p}\coloneqq L^{p',p}(\RR \times (0,T))$ and $\| \cdot \|_{p',p} \coloneqq \| \cdot \|_{L^{p',p}(\RR \times (0,T))}$.
This should not be confused with the weak-$L^p$ spaces, which we do not use in this article, except for the brief encounter in the proof of Theorem \ref{thm_est_nonhomo_BHE} below. This also should not be confused with some literature on the NSE where the order of the indices $p',p$ is switched; for example, the $L^{3,\infty }$ condition from \cite{ESS_2003} corresponds to $L^{\infty , 3}$ in our notation.
 
\subsection{Fractional Sobolev spaces}\label{sec_fractional_sobolev}

We denote by $\widehat{f}$ the Fourier transform (in the $x$ variable) of $f$, that is
\[
\widehat{f} (\xi ) \coloneqq \int_{\RR} f(x) \e^{-2\pi i x\xi} \d x, \qquad \text{ for }\xi \in \RR.
\]
 We will only consider Fourier transforms of functions that are bounded and have compact support. For such functions and any $s>0$ we denote by $\Lambda^s f$ the function with Fourier transform
\[
\widehat{\Lambda^{s} f} \coloneqq |\xi |^s \widehat{f}(\xi).
\]
Let 
\[
H^s \coloneqq \{ f\in L^2 (\RR ) \, :\, \Lambda^s f \in L^2 (\RR ) \}
\]
denote the fractional Sobolev space of order $s$ with the norm defined by
\[
\| f \|^2_{H^s} \coloneqq \int_\RR \left(1+ |\xi |^{2s} \right) |\widehat{f}(\xi)|^2\d \xi .
\] 
Observe that $\Lambda f\in L^2 $ if and only if $f_x\in L^2$ with
\eqnb\label{lambda_and_derivative}
2\pi \| \Lambda f \|_2 = \| f_x \|_2.
\eqne
Recall the Sobolev--Slobodeckij characterisation $H^s = W^{s,2}(\RR)$ for $s\in (0,1)$ with
\eqnb\label{sobolev_slobodeckij_characterisation}
\| f \|_{H^s} \simeq \| f \|_{W^{s,2}(\RR)},
\eqne
where ``$\simeq $'' denotes the equivalence of norms, $W^{s,2} (\Omega )\coloneqq \{ f\, : \, \| f \|_{W^{s,2}(\Omega )} <\infty \}$ and
\eqnb\label{norm_on_W^{s,2}}
\| f \|^2_{W^{s,2}(\Omega )} \coloneqq \left( \int_\Omega |f(x)|^2\d x + \int_\Omega \int_\Omega \frac{|f(x)-f(y)|^2}{|x-y|^{1+2s}} \d y\, \d x \right)
\eqne
for any open $\Omega \subset \RR$.
(For the proof of this characterisation see, for example, Proposition 3.4 in \cite{hitchhikers_guide}.)
\subsection{Biharmonic heat kernel}\label{sec_biharm_heat_kernel}
Let
\eqnb\label{BHK}
\Phi (x,t) \coloneqq \frac{c}{t^{1/4}} K(|x|/t^{1/4})
\eqne
where
\[
K(r) \coloneqq \int_0^\infty \e^{-s^4}\cos(r\,s)\,\d s
\]
and $c>0$ is such that $\| \Phi (t) \|_1 =1$. Then $\Phi$ is the biharmonic heat kernel (that is $w\coloneqq \Phi (t) \ast f$ satisfies the biharmonic heat equation $w_t+w_{xxxx}=0$ with initial condition $w(0)=f$, see \cite{ferrero_gazzola_grunau}). 
Note that, since $K(|x |)$ is smooth in $x\in \RR$ and has exponential decay as $|x|\to \infty$, we obtain
\eqnb\label{BHK_Lp_norm_of_derivatives}
\| \p_x^{(k)} \Phi (t) \|_p \leq C t^{-(k+1-1/p)/4}, \qquad k\geq 0.
\eqne
Moreover
\eqnb\label{BHK_Linty_of_FT_of_Phi}
\left| |\xi |^s \widehat{\Phi} (\xi ,t) \right| \leq C t^{-s/4}, \qquad s>0, t\in (0,T),
\eqne
for $\xi \in \RR$, $t>0$, $s\geq 0$, where $\widehat{\Phi }$ denotes the Fourier transform (with respect to $x$) of $\Phi$. The estimate \eqref{BHK_Linty_of_FT_of_Phi} (as well as \eqref{BHK_Lp_norm_of_derivatives}) follows directly from the formula \eqref{BHK} and from the smoothness of $K$. The estimate in \eqref{BHK_Lp_norm_of_derivatives} gives the following.
\begin{theorem}[Estimates for the convolution with the biharmonic heat equation]\label{thm_est_nonhomo_BHE}
If $f\in L^1_{\mathrm{loc}} (\RR \times [0,T ))$, $0\leq k \leq 3$ and 
\[
v(t)\coloneqq \int_0^t \Phi (t-s) \p_x^{(k)} f(s) \,\d s
\]
then
\[
\| v\|_{r',r} \leq C_{l,l',r,r'} \| f \|_{l',l},
\]
$l,l', r,r'$ satisfy $1\leq l \leq r \leq \infty$, $1\leq l' \leq r' \leq \infty$ and either 
\eqnb\label{exponents_sharp_<}
\frac{1}{l}+\frac{4}{l'} < \frac{1}{r}+\frac{4}{r'} + (4-k)
\eqne
or
\eqnb\label{exponents_not_sharp}
\frac{1}{l}+\frac{4}{l'} \leq  \frac{1}{r}+\frac{4}{r'} + (4-k)\quad \text{ and } 1<l'<r'< \infty .
\eqne
\end{theorem}
Note that the cases of $l',r'\in \{ 1,\infty\}$ and $l'=r'$ are allowed in \eqref{exponents_sharp_<}, but not in \eqref{exponents_not_sharp}.
\begin{proof}
We first focus on the case \eqref{exponents_sharp_<}. We have
\[
v(t)= \int_0^t \Phi (t-s) \ast \p_x^{(k)} f(s) \, \d s = (-1)^k \int_0^t \p_x^{(k)}  \Phi (t-s) \ast f(s) \, \d s.
\]
Thus, Young's inequality gives
\[
\| v(t) \|_r \leq \int_0^t \| \p_x^{(k)}  \Phi (t-s)\|_a \| f(s)\|_l \, \d s\leq C \int_0^t (t-s)^{-(k+1-1/a)/4}\| f(s)\|_l \, \d s ,
\]
where $1/r=1/a+1/l-1$. Hence
\eqnb\label{2nd_use_of_young}
\| v(t) \|_{r',r} \leq C \| f \|_{l',l} \| \, |s |^{-(k+1-1/a)/4} \|_{a'}, 
\eqne
where $1/{r'}=1/{a'}+1/{l'}-1$. The last norm is finite if and only if $a'(k+1-1/a)/4<1$, which is equivalent to \eqref{exponents_sharp_<}.

As for the case \eqref{exponents_not_sharp}, we need to use the Young inequality for weak spaces,
\eqnb\label{weak_young_inequality}
\| f\ast g\|_{r'} \leq \| f \|_{l'} \| g \|_{\widetilde{L}^{a'}},
\eqne
where $r',l',a'\in (1,\infty )$ are such that $1/r'=1/a'+1/l'-1$,
\[
\| g \|_{\widetilde{L}^{a'}}\coloneqq \inf \{ C>0 \, :\, d_g(\alpha ) \leq C^{a'}/\alpha^{a'} \quad \text{ for all }\alpha >0 \}
\]
denotes the norm of the weak-$L^{a'}$ space, and
\[
d_g (\alpha ) \coloneqq |\{ |g|>\alpha \} |, \quad \text{ for }\alpha >0
\]
denotes the distribution function of $g$. We refer the reader to Theorem 11.3 in \cite{McCormick_Robinson_Rodrigo_2013} for a proof of \eqref{weak_young_inequality}.

The point of using the weak form of Young's inequality \eqref{weak_young_inequality} is that $t^{-1/a'}$ is an element of the weak-$L^{a'}(0,1)$ space (but not of $L^{a'}(0,1)$), and so using it in the step leading to \eqref{2nd_use_of_young} gives
\[
\| v(t) \|_{r',r} \leq C \| f \|_{l',l} \| \, |s |^{-(k+1-1/a)/4} \|_{\widetilde{L}^{a'}}, 
\]
where $1/r'=1/a'+1/l'-1$ (note $a',r',l'\in(1,\infty)$ by \eqref{exponents_not_sharp}).
Thus the claim follows since the last norm is finite if and only if $a'(k+1-1/a)/4\leq l$, which is guaranteed by \eqref{exponents_not_sharp}.
\end{proof}
We will often consider a function $v:\RR \times [0,T)\to \RR$ of the form
\eqnb\label{conv_with_BHK_derivatives}
v(t) = \int_0^t \Phi (t-s) \ast g(s) \, \d s,
\eqne
where
\eqnb\label{form_of_g}
g=\sum_{k=0}^3 \phi_k \p_x^{k} f_k ,
\eqne
$f_k\in L^{l'}L^l$ (for some $l',l\geq 1$), $\phi_k \in C_0^\infty (Q)$ for some fixed $Q\Subset \RR\times (0,T)$.
Since \eqref{conv_with_BHK_derivatives} is not necessarily well-defined for such $f_k$'s (i.e. their derivatives might not exist), we will understand \eqref{conv_with_BHK_derivatives} as if all derivatives are transferred onto $\Phi$ and $\phi_k$'s (via integration by parts). Namely \eqref{conv_with_BHK_derivatives} means that
\[
v(t) = \sum_{k} \sum_{j=0}^k (-1)^k  {{k}\choose{j}} \int_0^t \p_{x}^{(j)}\Phi (t-s)\ast \left[ f_k(s) \p_x^{(k-j)} \phi_k (s)\right] \d s.
\]
We now formulate a corollary of Theorem \ref{thm_est_nonhomo_BHE} which is ``tailor-made'' for $v$'s of such form.
\begin{corollary}\label{cor_est_nonhomo_BHE_in_cylinder}
Let $v$ be given by \eqref{conv_with_BHK_derivatives} with $f_k\in L^{l_k',l_k}(Q)$, where $l_k',l_k$ satisfy \eqref{exponents_sharp_<} or \eqref{exponents_not_sharp} for some $r,r'$. Then 
\[
\| v\|_{r',r} \leq \sum_{k} C_k \| f_k \|_{L^{l_k',l_k}(Q)}.
\]
\end{corollary}
\begin{corollary}[Representation formula for \eqref{conv_with_BHK_derivatives}]\label{cor_repr_of_sol}
Suppose that $r',r\in [1,\infty ]$ and $w\in L^{r',r} $ is a distributional solution of
\[
w_t+w_{xxxx}= g\quad \text{ in } \RR \times (0,T)
\]
where $k\leq 3$, $g$ is given by \eqref{form_of_g} and each $f_k$ belongs to $L^{l_k',l_k}(Q)$ with $l_k',l_k$ satisfying \eqref{exponents_sharp_<} or \eqref{exponents_not_sharp}. Suppose further that $w=0$ in $\RR\times (0,t_0)$ for some $t_0 >0$. Then $w$ is given by the convolution with the biharmonic heat kernel, namely it satisfies the representation \eqref{conv_with_BHK_derivatives}.
\end{corollary}
\begin{proof}
Let 
\[
\widetilde{w} (t) \coloneqq \int_0^t \Phi (t-s) \ast g(s) \, \d s.
\]
Since $k\leq 3$ and $f_k \in L^{l_k'}L^{l_k}(Q)$, Corollary \ref{cor_est_nonhomo_BHE_in_cylinder} gives $\widetilde{w} \in L^{r',r}$. Moreover $\widetilde{w}$ satisfies $\widetilde{w}_t+\widetilde{w}_{xxxx}= g$, similarly as $w$. Thus $\widetilde{w}=w$ due to the uniqueness of solutions to the biharmonic heat equation, see the theorem below.
\end{proof}
\subsection{Uniqueness of solutions to homogeneous biharmonic heat equation}\label{sec_uniqueness_thm1}
\begin{theorem}[Uniqueness of solutions to homogeneous biharmonic heat equation]\label{thm_uniqueness_sol_biharm_heat}
Suppose that $q,q'\in [1,\infty ]$ and $v\in L^{q', q}$ is a distributional solution to the homogeneous biharmonic heat equation, that is
\eqnb\label{biharm_HE_uniqueness_equation}
\int_0^T \int_\RR v (\phi_t - \phi_{xxxx}) =0 \qquad \text{ for all }\phi \in C_0^\infty (\RR\times [0,T)).
\eqne
Then $v=0$.
\end{theorem}
\begin{proof}
We modify the argument from Section 4.4.2 of \cite{giga_saal}. We focus on the case $T<\infty$ (the case $T=\infty$ follows trivially by applying the result for all $T>0$). 

We first observe that the assumption $v\in L^{q',q}$ implies that \eqref{biharm_HE_uniqueness_equation} holds also for all $\phi \in C^\infty (\RR \times [0,T) )$ such that
\eqnb\label{wider_family_of_phi's}
\begin{cases}
 \p_t^k\p_x^m \phi \in L^{p',p} &\text{for every }k,m\geq 0, \\
\mathrm{supp }\, \phi \subset \RR \times [0,T' ) &\text{ for some } T' <T ,
\end{cases} 
\eqne
where $1/p+1/q=1$, $1/{p'}+1/{q'}=1$. Indeed, given such $\phi$ one can consider 
\[\phi_j (x,t)  \coloneqq \theta_j (x) \phi(x,t),\]
where $\theta_j (x) \coloneqq \theta (x/j)$ and $\theta \in C^\infty (\RR ; [0,1])$ is any function such that $\theta (\tau ) = 1$ for $|\tau |\leq 1$ and $\theta (\tau ) = 0 $ for $\tau \geq 2$. Then $\phi_j$ can be used as a test function in \eqref{biharm_HE_uniqueness_equation} and a simple use of the Dominated Convergence Theorem together with the properties \eqref{wider_family_of_phi's} and the assumption $v\in L^{q',q}$ proves the claim.

We will show that
\[
\int_0^T \int v \Psi   = 0
\]
for all $\Psi \in C_0^\infty (\RR \times (0,T) )$. The claim of the theorem then follows from the fundamental lemma of calculus of variations.

Given $\Psi \in C_0^\infty (\RR^3 \times (0,T) )$ let $T'<T$ be such that $\Psi(t) \equiv 0 $ for $t\geq T'$. Extend $\Psi$ by zero for $t\leq 0 $ and $t\geq T$. Let 
\[
\phi(x,t) \coloneqq - \int_0^{T-t} \Phi (T-t-s) \Psi (T-s)\, \d s,
\]
where $\Phi$ denotes the biharmonic heat kernel (see \eqref{BHK}).
Then $\phi$ solves the biharmonic heat equation backwards from $T$ with right-hand side $\Psi$, that is
\begin{equation}\label{phi_sol_of_backwards_heat}
\begin{cases}
\phi_t - \phi_{xxxx} = \Psi \qquad  \text{ in } \RR \times (-\infty,T),\\
\phi(T) =0.
\end{cases}
\end{equation}
In fact, we have $\phi(t)=0$ for $t\in [T',T]$. By the biharmonic heat estimates (see \eqref{BHK_Lp_norm_of_derivatives}) we see that $\phi$ satisfies \eqref{wider_family_of_phi's}. 
Thus \eqref{biharm_HE_uniqueness_equation} gives
\[
0=\int_0^T \int v (\phi_t - \phi_{xxxx} ) = \int_0^T \int v \Psi,
\]
as required.
\end{proof}
\subsection{Weak solutions of the surface growth model}\label{sec_weak_sols_to_SGM}
Here we define the notion of a weak solution to the surface growth model. Since the local Serrin condition is concerned with behaviour of weak solutions in bounded cylinders in space-time, we focus only on the notion of a local solution (in space-time; rather than a solution to the initial value problem).
\begin{definition}[Weak solution of the SGM]\label{def_of_weak_sol}
We say that $u\in L^{\infty ,2} (Q)$ is a \emph{weak solution} of the surface growth model on a cylinder $Q$ if $u_{xx} \in L^{2,2} (Q)$ and
\[
 \int_Q \left( u \, \phi_t - u_{xx} \phi_{xx} - u_x^2 \phi_{xx} \right) =0
\]
for all $\phi \in C_0^\infty (Q)$. 
\end{definition}
In what follows we will assume that $u$ is a weak solution of the SGM on a given $Q$. Note that any weak solution $u$ on a cylinder $Q$ satisfies 
\eqnb\label{reg_of_any_weak_sol}
 u_x \in L^{10/3,10/3} (Q) \eqne
(which can be shown using Sobolev interpolation; see (2.5) in \cite{Ozanski_Robinson_2017} for details), and so in particular the integral in the equation above is well-defined. Moreover
\eqnb\label{reg_for_q_between_1_2}
 u_x \in L^{16/3,2} (Q), \eqne
 given the main assumption of this article  (i.e. \eqref{serrin_cond_SGM}) is satisfied with $q\in (1,2]$, which follows by a simple application of Lebesgue interpolation (in space and then in time) between \eqref{reg_of_any_weak_sol} and \eqref{serrin_cond_SGM}. We will need \eqref{reg_for_q_between_1_2} in order to circumvent a certain technical issue in the proof of the main result when $q\in (1,2]$ (see the comments following \eqref{cor_applied_to_w} for details).

\section{Proof of the main result}\label{sec_local_serrin_cond}
Here we show that if $Q\Subset \RR \times (0,T)$ and $u_x \in L^{q',q} (Q)$ for any $q,q'\in (1,\infty ] $ such that either
\eqnb\label{exponents_Serrin}
\frac{1}{q}+\frac{4}{q'} < 1 \quad \text{ or } \quad 
\frac{1}{q}+\frac{4}{q'} = 1
\eqne
then $u\in C^\infty (Q)$.

The main idea of the proof is to note that the function 
\[ v\coloneqq u_x\]
satisfies the equation
\[
v_t+v_{xxxx}=-\p_{xxx} v^2
\]
in $Q$, and that one can apply the estimate from Theorem \ref{thm_est_nonhomo_BHE} to increase the regularity of $v$. In order to apply this strategy, we need to extend $v$ to $\RR\times (0,T)$ and explore the representation of such an extension by a formula similar to \eqref{conv_with_BHK_derivatives}. To be precise, given a cutoff function $\phi\in C_0^\infty (Q;[0,1])$, let
\[
w\coloneqq v\phi .
\]
Then $w$ satisfies
\eqnb\label{eq_for_w}
w_t + w_{xxxx} =  -(wv)_{xxx} +f_v \qquad \text{ in } \RR\times (0,T),\eqne
where
\[
f_v \coloneqq (v\,\phi_t +4v_{xxx} \phi_x + 6 v_{xx}\phi_{xx} +4v_x \phi_{xxx}+v\, \phi_{xxxx}) + (3\phi_x \p_{xx}v^2 + 3\phi_{xx}\p_x v^2 + \phi_{xxx}v^2 ),
\]
which we will write more concisely as
\eqnb\label{what_is_fv}
f_v=\sum_{m=0}^3  \phi_m \p_x^{m} v  +\sum_{k=0}^2 \psi_k \p_x^{k} (v^2) 
\eqne
for some $\psi_k,\phi_m \in C_0^\infty (Q)$ (each being a constant multiple of a derivative of $\phi$).

We note that $w$ satisfies the representation formula
\eqnb\label{repr_of_w}
w(t) = \int_0^t \Phi_{xxx}(t-s) \ast \left[ w(s)v(s)\right] \d s + \int_0^t \Phi(t-s) \ast f_v(s)\,\d s\quad \text{ for }t\in (0,T),
\eqne
where the last term is understood in the same sense as \eqref{conv_with_BHK_derivatives}. Indeed, since $v\in L^{10/3} (Q)$ (recall \eqref{reg_of_any_weak_sol}) we see that $wv, v^2 \in L^{5/3} (Q)$ and $w\in L^{10/3} (Q)\subset L^{5/3} (Q)$ (in particular $w\in L^{5/3,5/3}$ as $w=0$ outside $Q$), and so we can use Corollary \ref{cor_repr_of_sol} (since the choice $r=r'=5/3$, $l=l'=5/3$ satisfies \eqref{exponents_sharp_<} trivially for any $k\in \{ 0,1,2,3\}$) to obtain \eqref{repr_of_w}.

\subsection{The subcritical case}\label{sec_subcritical}
Here we focus on the subcritical case, namely the first case of \eqref{exponents_Serrin}.
\begin{theorem}[Local Serrin condition, subcritical case]\label{thm_subcritical_case}
Let $u$ be a weak solution to the surface growth model on a cylinder $Q$ and let $q,q' \in (1,\infty ]$ satisfy $1/q+4/q'<1$. If
\[
u_x \in L^{q',q} (Q)
\]
then $u\in C^\infty (Q)$.
\end{theorem}
\begin{proof}
First, by translation we can assume that $Q$ is contained within a time interval $(0,T)$ for some $T>0$; that is $Q\Subset \RR\times (0,T)$.

Secondly we can assume that $q'<\infty$. Indeed, since the choice of the exponents $q',q$ is subcritical, the case $q'=\infty$ can be reduced to $q'<\infty $ by a use of H\"older's inequality,
\[
\| u_x \|_{L^{q',q}(Q)} \leq C \| u_x \|_{L^{\infty,q}(Q)},
\] 
where $q'<\infty$ is sufficiently large such that $1/q+4/q'<1$ holds.

Thirdly it suffices to prove the result under a smallness condition  
\eqnb\label{smallness_condition_sharp_case}
\| u_x \|_{L^{q',q}(Q)}\leq \delta
\eqne
for $\delta >0$ sufficiently small such that
\eqnb\label{smallness_condition_bound_on_delta}
\delta < \frac{1}{2} \max \left( C_{q,q',\infty, \infty} , C_{q,q',q/2,q'/2}\right)^{-1} ,
\eqne
where the constants on the right-hand side are from Theorem \ref{thm_est_nonhomo_BHE}.
Indeed, since $q'<\infty$, $\| u_x \|_{L^{q',q}(\widetilde{Q})}<\delta$ for every sufficiently small subcylinder $\widetilde{Q}$ of $Q$. Thus if $u\in C^\infty (\widetilde{Q})$ for every such cylinder, then the same is true for $Q$.

The proof of the result proceeds in a few steps.\vspace{0.7cm}\\
\emph{Step 1.} Show that $u_x \in L^\infty (\widetilde{Q})$ for any $\widetilde{Q}\Subset Q$.\\

Let $v\coloneqq u_x$ and let $\phi\in C_0^\infty (Q;[0,1])$ be such that $\phi=1$ on $\widetilde{Q}$. By the representation \eqref{repr_of_w} Corollary \ref{cor_est_nonhomo_BHE_in_cylinder} applied with $r=r'=\infty$ gives
\eqnb\label{cor_applied_to_w}
\| w \|_{\infty , \infty}\leq C_{q,q',\infty,\infty } \| w v \|_{q',q} + C \left( \| v^2 \|_{L^{q'/2,q/2}(Q)} + \| v \|_{L^{q',q}(Q)} \right)
\eqne
Here we took $k=3$, $l=q$, $l'=q'$ for the main nonlinearity (i.e. the term $-(wv)_{xxx}$ in \eqref{eq_for_w}) as well as $l=q/2$, $l'=q'/2$ for the quadratic terms (i.e. the terms involving $v^2$ in \eqref{what_is_fv}) and $l=q$, $l'=q'$ for the linear terms (i.e. the terms involving $v$ in \eqref{what_is_fv}). Note that here we have implicitly assumed that $q\geq 2$ (while $q'\geq 2$ follows from the assumption); the case $q\in (1,2)$ can be reduced to the case $q\geq 2$ by exploiting \eqref{reg_for_q_between_1_2}, which we explain in detail in Section \ref{sec_the_case_q_1_2} below. 

Applying H\"older's inequality to the first term on the right-hand side of \eqref{cor_applied_to_w} gives
\[
\| w v \|_{q',q} \leq \| w  \|_{\infty , \infty } \| v \|_{L^{q',q}(Q)}
\]
Thus given the smallness condition \eqref{smallness_condition_bound_on_delta} we can absorb this term on the left hand side to see that
\eqnb\label{w_in_Linfty_claim}
w\in L^\infty (Q)
\eqne
and so in particular $v\in L^\infty (\widetilde{Q})$.
Note however, that there is a gap in this step, since subtracting $\| w \|_{L^\infty (Q)}$ we have implicitly assumed that this norm is finite. If it is not the case then such argument may have a potentially fatal flaw. 

This gap can be dealt with by a rather technical procedure of regularising the equation \eqref{eq_for_w} and taking a limit of the solutions to the regularised equations, which we explain in more detail in the next step.\vspace{0.7cm}\\
\emph{Step 1'.} Verify \eqref{w_in_Linfty_claim}\\

Extend $v$ by zero outside $Q$. Let $d\coloneqq \inf \{ t \, : \, (x,t) \in Q \}>0$. For every $\varepsilon \in (0,d/2 )$ let $w^\varepsilon\in L^{\infty , \infty }$ be a solution of the problem
\eqnb\label{to_show_for_w^epsilon}
\begin{cases}
w_t^\varepsilon + w_{xxxx}^\varepsilon= - \p_{xxx} (v_\varepsilon w^\varepsilon ) + f_{v_\varepsilon}\qquad \text{ in } \RR\times (0,T),\\
w^\varepsilon(0)=0,
\end{cases}
\eqne
where $v_\varepsilon$ denotes the mollification of $v$ (in both space and time), and the initial condition is understood in the sense that $w=0$ in $\RR\times (0,c)$ for some $c>0$ independent of $\varepsilon$ (one can take for example $c\coloneqq d$).
The existence of such a $w^\varepsilon$ follows from the Picard iteration, which we now briefly outline. Let 
\[
w_0^\varepsilon (t) \coloneqq  \int_0^t \Phi (t-s) \ast f_{v_\varepsilon} (s) \, \d s,\qquad \text{ for } t>0,
\] 
and then set
\[
w_{m+1}^\varepsilon (t) \coloneqq  \int_0^t \Phi_{xxx}(t-s)\ast [v_\varepsilon (s) w_m^\varepsilon (s) ]\, \d s + w_0^\varepsilon (t),\qquad \text{ for } t>0, m=0,1,\ldots
\]
Since $v_\varepsilon, f_{v_\epsilon } \in C_0^\infty (\RR \times (0,T))$, Corollary \ref{cor_est_nonhomo_BHE_in_cylinder} gives that $w_m^\varepsilon \in L^{\infty , \infty }$ for each $m$. Moreover each $w^\varepsilon_m$ satisfies the equation
\eqnb\label{equation_for_w_k^epsilon}
\p_t w_m^\varepsilon + \p_{xxxx} w_m^\varepsilon = -\p_{xxx} (v_\varepsilon w_{m-1}^\varepsilon ) + f_{v_\varepsilon} \quad \text{ in }\RR\times (0,T)
\eqne
and $w^\varepsilon_m (t)=0$ for $t<d/2$. By Corollary \ref{cor_est_nonhomo_BHE_in_cylinder}
\[\begin{split}
\| w_{m+1}^\varepsilon - w_m^\varepsilon \|_{\infty, \infty} &\leq C_{q,q',\infty, \infty} \| v_\varepsilon (w_m^\varepsilon-w_{m-1}^\varepsilon)\|_{q',q}\leq C_{q,q',\infty, \infty} \| v_\varepsilon \|_{q',q} \| w_m^\varepsilon-w_{m-1}^\varepsilon\|_{\infty ,\infty} \\
&\leq C_{q,q',\infty, \infty} \| v \|_{q',q} \| w_m^\varepsilon-w_{m-1}^\varepsilon\|_{\infty ,\infty} \leq \frac{1}{2}\| w_m^\varepsilon-w_{m-1}^\varepsilon\|_{\infty, \infty},
\end{split}
\]
where also used H\"older's inequality, the fact that mollification does not increase $L^p$ norms and the smallness assumption \eqref{smallness_condition_sharp_case}. 
Thus $\{ w^\varepsilon_m \}$ is a Cauchy sequence in $L^{\infty,\infty }$ and so
\[
w^\varepsilon_m \to w^\varepsilon \quad \text{ in } L^{\infty,\infty }\text{ as }m\to \infty
\]
for some $w^\varepsilon \in L^{\infty , \infty}$ such that $w^\varepsilon (t)=0$ for $t<d/2$. Taking the limit $m\to \infty$ in \eqref{equation_for_w_k^epsilon} gives \eqref{to_show_for_w^epsilon} (recall we mean partial differential equations in the distributional sense), which concludes the proof of the existence of $w^\varepsilon$.

We conclude this step by showing that $w^\varepsilon \stackrel{\ast }{\rightharpoonup} w$ as $\varepsilon \to 0$ (on some subsequence) in $L^{\infty ,\infty }$ (and so in particular $w\in L^{\infty,\infty}$, as required). 

Note that $w^\varepsilon$ satisfies the representation
\[
{w^\varepsilon} (t) \coloneqq - \int_0^t \Phi_{xxx} (t-s) \ast [v_\varepsilon (s) w^\varepsilon (s) ] \d s + w_0^\varepsilon (t)
\]
(cf. \eqref{repr_of_w}). Thus Corollary \ref{cor_est_nonhomo_BHE_in_cylinder} gives
\[\begin{split}
\| w^\varepsilon \|_{\infty, \infty} &\leq C_{q,q',\infty, \infty} \| v_\varepsilon \|_{q',q} \|w^\varepsilon \|_{\infty, \infty } + C \left( \| v_\varepsilon^2 \|_{q'/2,q/2}+ \| v_\varepsilon \|_{q',q} \right)\\
&\leq C_{q,q',\infty, \infty} \| v \|_{q',q} \|w^\varepsilon \|_{\infty,\infty } + C \left( \| v \|_{q',q}^2+ \| v \|_{q',q} \right),
\end{split}
\]
as in \eqref{cor_applied_to_w}.
Therefore the smallness condition \eqref{smallness_condition_sharp_case} lets us absorb the first term on the right hand side to obtain
\[
\| w^\varepsilon \|_{\infty, \infty}  \leq C 
\]
for all $\varepsilon >0$. In the same way one obtains
\[
\| w^\varepsilon \|_{q',q}  \leq C 
\]
Hence there exists a sequence $\varepsilon_k\to 0$ and $\widetilde{w} \in L^{\infty, \infty} \cap L^{q',q}$ such that $w^{\varepsilon_k} \stackrel{\ast }{\rightharpoonup} \widetilde{w} $ in $L^{\infty , \infty}$ and $w^{\varepsilon_k} {\rightharpoonup} \widetilde{w} $ in $L^{q',q}$. Since $v\in L^{2,2}$ we see that $v_{\varepsilon_k} \to v$ in $L^{2,2}$ (as a property of the mollification operation) and thus we can take the limit in the partial differential equation in \eqref{to_show_for_w^epsilon} (in the sense of distributions) to obtain that $\widetilde{w}$ and $w$ satisfy the same partial differential equation and are both elements of $L^{q',q}$. 
 It remains to show that $w=\widetilde{w}$. Let $h\coloneqq w-\widetilde{w}$. Then $h\in L^{q',q}$, $h=0$ in $\RR\times (0,c)$ for some $c>0$ (since the same is true for both $w$, $\widetilde{w}$) and
 \[
 h_t+h_{xxxx}=-\p_{xxx}(vh).
 \]
As in \eqref{repr_of_w}, $h$ satisfies the representation formula,
\[
h(t)=-\int_0^t \Phi_{xxx} (t-s) \ast [v(s)h(s)]\d s.
\]
Thus Corollary \ref{cor_est_nonhomo_BHE_in_cylinder} gives
\[
\| h \|_{q',q} \leq C_{q,q',q/2,q'/2} \| v \,h\|_{q'/2,q/2} \leq C_{q,q',q/2,q'/2} \| v\|_{q',q} \| h \|_{q',q} \leq \frac{1}{2} \| h \|_{q',q},
\]
where we assumed that $\delta$ from the smallness condition \eqref{smallness_condition_sharp_case} satisfies $\delta <C_{q,q',q/2,q'/2}/2$. Thus, since $\| h \|_{q',q}<\infty$, the above inequality implies $h=0$, as required.\vspace{0.7cm}\\
\emph{Step 2.} Show boundedness of higher derivatives in $x$.\\

We proceed by induction. We will set $v^{(k)}\coloneqq\p_x^{k} v$ for brevity. We will show that if $v^{(k)}\in L^{\infty, \infty} (Q)$ then $v^{(k+1)}\in L^{\infty , \infty} (\widetilde{Q})$ for any subcylinder $\widetilde{Q}\Subset Q$.

Let $M>1$ be such that $\| v\|_{L^{\infty , \infty }(Q)} , \ldots , \| v^{(k)}\|_{L^{\infty , \infty }(Q)} <M$ and let $\phi \in C_0^\infty (Q';[0,1])$ be such that $\phi=1 $ on $Q''$ for some cylinders $Q',Q''$ such that $\widetilde{Q}\Subset Q''\Subset Q' \Subset Q$. Let 
\[
z\coloneqq v^{(k+1)}\phi .
\]
Then $z$ satisfies 
\[
\begin{split}
z_t+z_{xxxx} &=-(\p_x^{k+4}v^2)\phi + 4v^{(k+4)} \phi_x + f_1\\
&= -(\p_x^{k+4}v^2 )\phi + 4(v^{(k+1)}\phi_x)_{xxx}   + f_1+f_2\\
&=  -2(v^{(k+1)}v)_{xxx} \,\phi + 4(v^{(k+1)}\phi_x)_{xxx}   + f_1+f_2+f_3\\
&= -2(v\,z )_{xxx} + 4(v^{(k+1)}\phi_x)_{xxx}   + f_1+f_2+f_3 +f_4
\end{split}
\]
in $\RR \times (0,T)$ (in the sense of distributions), where
\[
\begin{cases}
f_1\coloneqq v^{(k+1)}\phi_t + 6v^{(k+3)}\phi_{xx} +4v^{(k+2)} \phi_{xxx} + v^{(k+1)}\phi_{xxxx},\\
f_2 \coloneqq -12v^{(k+3)}\phi_{xx}-12 v^{(k+2)} \phi_{xxx}-4v^{(k+1)} \phi_{xxxx},\\
f_3\coloneqq -\sum_{j=1}^{k} {k+1 \choose j } \left( v^{(j)}v^{(k+1-j)} \right)_{xxx} \phi,\\
f_4\coloneqq 6 (v^{(k+1)}v)_{xx} \phi_x+6 (v^{(k+1)}v)_x \phi_{xx} + 2 v^{(k+1)}v\,\phi_{xxx}.
\end{cases}
\]
In short, 
\eqnb\label{eq_for_z_higher_derivs}
z_t+z_{xxxx}=-2(v\,z )_{xxx} + 4(v^{(k+1)}\phi_x)_{xxx}   + F_{k+1},
\eqne
where 
\[
F_{k+1} \coloneqq f_1+f_2+f_3+f_4
\]
consists of (at most) second order derivatives (with respect to $x$) of (linear or quadratic) terms which include $v,\ldots , v^{(k+1)}$ (and a derivative of the cutoff function $\phi$).

Suppose for the moment that 
\eqnb\label{l2_assumption}
v^{(k+1)}\in L^{2,2} (Q').
\eqne
Then $z\in L^{2,2}$ and $vz, v^{(k+1)}\phi_x\in L^{2,2}$ and so Corollary \ref{cor_repr_of_sol} gives the representation formula for $z$,
\eqnb\label{repr_of_z}\begin{split}
z(t) =& 2\int_0^t \Phi_{xxx} (t-s)\ast [v(s)z(s)]\d s -4\int_0^t \Phi_{xxx} (t-s) \ast [v^{(k+1)}(s)\phi_x(s)]\d s \\
&+ \int_0^t \Phi (t-s)\ast F_{k+1}(s)\d s,\qquad t\in (0,T),
\end{split}
\eqne
where the last term is understood in the same sense as \eqref{conv_with_BHK_derivatives}. Recalling that $z=v^{(k+1)}$ in $Q''$ and using Corollary \ref{cor_est_nonhomo_BHE_in_cylinder} gives
\eqnb\label{bootstraping_of_Lp_for_z}\begin{split}
\| v^{(k+1)} \|_{L^{r',r}(Q'')} \leq \| z \|_{r',r}&\leq C_\phi (\| v z \|_{l',l} + \| v^{(k+1)} \phi_x \|_{l',l} + M^2  )\\
&\leq C_\phi M (\| v^{(k+1)} \|_{L^{l',l}(Q')} + M^2 )
\end{split}
\eqne
whenever $r,r',l,l'\in [1,\infty ]$ satisfy
\[
\frac{1}{l}+\frac{4}{l'}<\frac{1}{r}+\frac{4}{r'}+1.
\] 
In other words we have obtained an increase in the integrability of $v^{(k+1)}$ with the cost of shrinking the domain slightly. It remains to bootstrap the inequality in \eqref{bootstraping_of_Lp_for_z}. Namely let $Q'''$ be a cylinder such that 
\[
\widetilde{Q}\Subset Q''' \Subset Q'' \Subset Q'
\]
and let $\phi,\phi',\phi'' \in C_0^\infty (Q';[0,1])$ be the cutoff functions such that $\phi'$ cuts-off $Q'''$ in $Q''$ (i.e. $\phi'=1$ on $Q'''$ and $\phi'=0$ outside $Q''$) and $\phi''$ cuts-off $\widetilde{Q}$ in $Q'''$. Then apply \eqref{bootstraping_of_Lp_for_z} with $l'=l=2$, $r'=r=3$ to obtain
\[
\| v^{(k+1)} \|_{L^{3,3}(Q'')} \leq C_\phi M (\| v^{(k+1)} \|_{L^{2,2}(Q')} + M^2 ).
\]
Taking $l'=l=3$, $r'=r=7$ we obtain
\[
\| v^{(k+1)} \|_{L^{7,7}(Q''')} \leq C_{\phi'} M (\| v^{(k+1)} \|_{L^{3,3}(Q'')} + M^2 ).
\]
Finally the choice $l'=l=7$, $r'=r=\infty$ gives
\[
\| v^{(k+1)} \|_{L^{\infty,\infty }(\widetilde{Q})} \leq C_{\phi''} M (\| v^{(k+1)} \|_{L^{7,7}(Q''')} + M^2 ),
\]
as required. Therefore, in order to complete this step, it remains to verify \eqref{l2_assumption}.\vspace{0.7cm}\\
\emph{Step 2'.} Verify \eqref{l2_assumption}.\\

The case $k=0$ follows from the definition of a weak solution (see Definition \ref{def_of_weak_sol}; recall also that $v=u_x$). In the case $k\geq 1$ let $\phi \in C_0^\infty (Q;[0,1])$ be such that $\phi =1$ on a cylinder $\mathcal{Q}$ such that $Q' \Subset \mathcal{Q} \Subset Q$ and set
\[
\eta \coloneqq v^{(k)}\phi.
\]
As in \eqref{eq_for_z_higher_derivs}, $\eta $ satisfies
\eqnb\label{eq_for_eta_high_derivs}
\eta_t+\eta_{xxxx}=-2(v\,\eta )_{xxx} + 4(v^{(k)}\phi_x)_{xxx}   + F_k,
\eqne
where $F_k$ consists of (at most) second order derivatives (with respect to $x$) of terms consisting of bounded functions $v,\ldots ,v^{(k)}$ (multiplied by a derivative of the cutoff function $\phi$). Namely
\[
F_k = \phi_k^{(2)} \p_{xx} F_k^{(2)} + \phi_k^{(1)}\p_{x} F_k^{(1)} + \phi_k^{(0)} F_k^{(0)}  ,
\]
where $F_k^{(0)}, F_k^{(1)}, F_k^{(2)} \in L^{\infty,\infty}(Q)$ and $\phi_k^{(l)} \in C_0^\infty (Q)$, $l=0,1,2$. Now observe that we can absorb the functions $\phi_k^{(l)}$ into $F_k^{(l)}$, $l=0,1,2$, by the chain rule. Namely there exist $G_k^{(0)}, G_k^{(1)}, G_k^{(2)} \in L^{\infty,\infty}(Q)$ such that
\[
F_k =  \p_{xx} G_k^{(2)} + \p_{x} G_k^{(1)} + G_k^{(0)}  .
\]

As in \eqref{repr_of_z}, $\eta$ satisfies the representation
\eqnb\label{repr_of_eta}
\begin{split}
\eta (t) &= 2\int_0^t \Phi_{xxx} (t-s)\ast [v(s)\eta (s)]\d s -4\int_0^t \Phi_{xxx} (t-s) \ast [v^{(k)}(s)\phi_x(s)]\d s \\
&+ \int_0^t \Phi (t-s)\ast F_{k}(s)\d s,\qquad t\in (0,T).
\end{split}
\eqne
At this point we pause for a moment and comment on our strategy in an informal way. Formally, one could take the $x$-derivative in \eqref{repr_of_eta} to obtain
\[
\begin{split}
\eta_x (t) &= 2\int_0^t \Phi_{xxx} (t-s)\ast [v_x(s)\eta (s)+v(s)\eta_x (s)]\d s \\
&-4\int_0^t \Phi_{xxx} (t-s) \ast [v^{(k+1)}(s)\phi_x(s)+v^{(k)}(s)\phi_{xx}(s)]\d s \\
&+ \int_0^t \Phi_x (t-s)\ast F_{k}(s)\d s,\qquad t\in (0,T).
\end{split}
\]
We would now like to use the estimates from Corollary \ref{cor_est_nonhomo_BHE_in_cylinder} to obtain an estimate on $\| \eta_x \|_{2,2}$ and so deduce that $v^{(k+1)}\in L^{2,2}(\mathcal{Q})$. In fact, the terms including $v_x\eta$, $v^{(k)}\phi_{xx}$ and $F_k$ could be dealt with easily, since $v, \ldots , v^{(k)} \in L^{\infty , \infty }(Q)$ (and $k\geq 1$), and the term including $v \eta_x$ could be dealt with by using the smallness condition \eqref{smallness_condition_sharp_case} as in step 1'. However, the term including $v^{(k+1)}\phi_x$ (that is the one originating from the linear part, cf. the definition of $f_1$) cannot be dealt with in this way (as at this point we know nothing about $v^{(k+1)}$). 

The way to deal with this problem is to increase the regularity of $\eta$ by a ``half of the derivative in $x$'' at a time. Namely we will first show that $\Lambda^{1/2}\eta \in L^{2,2}$ and then deduce that $\Lambda \eta' \in L^{2,2}$, where 
\eqnb\label{def_of_eta'}
\eta' \coloneqq v^{(k)} \phi'
\eqne
and $\phi' \in C_0^\infty (\mathcal{Q};[0,1])$ is such that $\phi'=1$ on $Q'$. Thus, since $2\pi\| \Lambda \eta' \|_{2,2}=\| \p_x \eta' \|_{2,2}$ (recall \eqref{lambda_and_derivative}) we will obtain $\p_x \eta' \in L^{2,2}$, and so in particular $v^{(k+1)} \in L^{2,2}(Q')$. 

Taking the Fourier transform (in $x$) of \eqref{repr_of_eta} we obtain
\eqnb\label{repr_of_eta_fourier}
\begin{split}
\widehat{\eta }(t) =& -16 \pi^3 i\int_0^t \xi^3 \widehat{\Phi}  (t-s) \widehat{v\eta }(s)\d s +32 \pi^3 i \int_0^t \xi^3 \widehat{\Phi} (t-s) \widehat{v^{(k)}\phi_x}(s)\d s \\
&+\sum_{m=0}^2 (2\pi i)^m \int_0^t \xi^{m} \widehat{\Phi } (t-s) \widehat{G_{k}^{(m)}}(s)\d s \qquad t\in (0,T).
\end{split}
\eqne
Multiplying \eqref{repr_of_eta_fourier} by $|\xi|^s$, where $s=s_1+s_2$, $s_1,s_2\in [ 0,1)$, taking the $L^2$ norm (in $\xi$) and using Plancherel's property we obtain
\[
\begin{split}
\| \Lambda^s \eta (t) \|_2 \leq &C \int_0^t \| \xi^{3+s_1} \widehat{\Phi } (\xi ,t-s) \|_{\infty } \| \xi^{s_2} \widehat{v\eta }(\xi, s) \|_2 \d s\\
&+C \int_0^t \| \xi^{3+s_1} \widehat{\Phi } (\xi ,t-s) \|_{\infty } \| \xi^{s_2} \widehat{v^{(k)} \phi_x }(\xi, s) \|_2 \d s\\
&+C \sum_{m=0}^2 \int_0^t \| \xi^{m+s} \widehat{\Phi } (\xi ,t-s) \|_{\infty } \| G_k^{(m)} (s) \|_2 \d s\\
\leq & C \int_0^t   \frac{\| \Lambda^{s_2} (v(s) \eta (s) ) \|_2+\| \Lambda^{s_2} (v^{(k)}(s) \phi_x (s) ) \|_2}{(t-s)^{(3+s_1)/4}} \d s+ C_T\int_0^t \frac{\| G_k(s)\|_2}{(t-s)^{(2+s)/4}} \d s,
\end{split}
\]
where $G_k =|G_k^{(0)}|+|G_k^{(1)}|+|G_k^{(2)}|$ and we also used \eqref{BHK_Linty_of_FT_of_Phi} as well as applied the integral version of the Minkowski inequality. Taking the $L^2$ norm in time and using Young's inequality for convolutions we obtain
\eqnb\label{lambda^s_estimate}
\| \Lambda^s \eta \|_{2,2}\leq C_{s_1,s_2,\phi} \left( \|\Lambda^{s_2} (v\eta ) \|_{2,2} + \|\Lambda^{s_2} (v^{(k)} \phi_x ) \|_{2,2} + \| G_k \|_{2,2} \right)
\eqne
for $s_1,s_2\in [0,1)$, $s=s_1+s_2$. Taking $s_1=1/2$, $s_2=0$ we see that $\Lambda^{1/2} \eta \in L^{2,2}$. Thus $\eta \in L^2 ((0,T);H^{1/2})$ and, thanks to the Sobolev--Slobodeckij characterisation \eqref{sobolev_slobodeckij_characterisation}, 
\[ \eta \in L^2 ((0,T); W^{1/2,2}(\RR )),\]
that is
\[
\int_0^T \int_\RR \int_\RR \frac{|\eta (x,t)-\eta (y,t)|^2}{|x-y|^{2}} \d y\, \d x\, \d t <\infty .
\]
Restricting the time domain to $\mathcal{I}$ and spatial domain to $\mathcal{B}$, where $\mathcal{Q}=\mathcal{B} \times \mathcal{I}$, we obtain that
\eqnb\label{tempus}
\int_\mathcal{I} \int_\mathcal{B} \int_\mathcal{B} \frac{|v^{(k)} (x,t)-v^{(k)} (y,t)|^2}{|x-y|^{2}} \d y\, \d x\, \d t <\infty .
\eqne
Thus
\[
v^{(k)} \in L^2 (\mathcal{I} ; W^{1/2,2}( \mathcal{B} ))
\]
Now letting $\eta'$ be the cutoff of $v^{(k)}$ as in \eqref{def_of_eta'} we can apply the triangle inequality and \eqref{tempus} to see that both $v \eta'$ and $v^{(k)} \phi'$ belong to $L^2 (\mathcal{I} ; W^{1/2,2}( \mathcal{B} ))$ as well (recall that $v_x \in L^{\infty , \infty }$, since $k\geq 1$). Thus, since $v \eta'$ and $v^{(k)} \phi'$ are supported within $\mathcal{Q}$, they belong to $L^2 ((0,T); W^{1/2,2}( \RR))=L^2 ((0,T); H^{1/2})$ (by the Sobolev--Slobodeckij characterisation \eqref{sobolev_slobodeckij_characterisation}), and so 
\[
\|\Lambda^{1/2} (v\eta' ) \|_{2,2}, \|\Lambda^{1/2} (v^{(k)} \phi'_x ) \|_{2,2} <\infty.
\]
Therefore, we can use \eqref{lambda^s_estimate} (applied to $\eta'$, rather than $\eta$) with $s_1=s_2=1/2$ to obtain
\[
\| \eta'_x \|_{2,2} =2\pi \| \Lambda \eta' \|_{2,2} < \infty,
\]
where we have also recalled \eqref{lambda_and_derivative}. In particular $\| v^{(k+1)} \|_{L^{2,2} (Q')} = \| \eta'_x \|_{L^{2,2} (Q')} < \infty$, as required. \vspace{0.7cm}\\
\emph{Step 3.} Deduce the smoothness of $u$.\\

From the surface growth equation, $u_t+u_{xxxx}+\p_{xx} u_x^2=0$, and steps 1 and 2 we see that $u_t$ (in the sense of weak derivatives) is bounded on every compact subset of $Q$. Similarly every derivative (in both $x$ and $t$) is bounded on every compact subset of $Q$. Thus the Rellich--Kondrachov embedding (see, for example, Theorem 6.3 in \cite{adams_fournier}) gives smoothness of $u$ in $Q$.\end{proof}

\subsection{The critical case}\label{sec_critical}
We now focus on the critical case, namely the second case of \eqref{exponents_Serrin}. 
\begin{theorem}[Local Serrin condition, critical case]\label{thm_critical_case}
Let $u$ be a weak solution to the surface growth model on a cylinder $Q$ and let $q,q' \in (1,\infty ]$ satisfy $1/q+4/q'=1$. If
\[
u_x \in L^{q',q} (Q)
\]
then $u\in C^\infty (Q)$.
\end{theorem}
\begin{proof}
Similarly as in Theorem \ref{thm_subcritical_case} we can assume that $Q\Subset \RR\times (0,T)$ and we can assume the smallness condition \eqref{smallness_condition_sharp_case}, since $q'<\infty $.

Let $\widetilde{Q}\Subset Q$ and fix $p' \in (q',\infty )$, $p \in (q,\infty )$ satisfying $1/p+4/p'<1$. We will show that $u_x\in L^{p',p} (\widetilde{Q})$  (then the claim follows from Theorem \ref{thm_subcritical_case}).
 
As in the proof of Theorem \ref{thm_subcritical_case} let $v\coloneqq u_x$, $\phi\in C_0^\infty (Q;[0,1])$ be such that $\phi=1$ on $\widetilde{Q}$. As in \eqref{repr_of_w} $w\coloneqq v\phi$ satisfies the representation \eqref{repr_of_w},
\eqnb\label{repr_of_w_again}
w(t) = -\int_0^t \Phi_{xxx}(t-s) \ast \left[ w(s)v(s)\right] \d s + \int_0^t \Phi(t-s) \ast f_v(s)\,\d s.
\eqne
 Corollary \ref{cor_est_nonhomo_BHE_in_cylinder} applied with $r=p$, $r'=p'$ gives
 \eqnb\label{cor_applied_to_w_critical}
\| w \|_{p' , p}\leq C_{l,l',p,p' } \| w v \|_{l',l} + C \left( \| v^2 \|_{L^{q'/2,q/2}(Q)} + \| v \|_{L^{q',q}(Q)} \right),
\eqne
where $1/l=1/p+1/q$, $1/l'=1/p'+1/q'$ (note $l\in (q,p)$, $l'\in (q',p')$), cf. \eqref{cor_applied_to_w}. (Note that again we have implicitly assumed that $q>2$, see the section below for the case $q\in (1,2]$.) Applying H\"older's inequality to first term on the right-hand side of \eqref{cor_applied_to_w_critical} and using the smallness condition \eqref{smallness_condition_sharp_case} with $\delta < 1/2 C_{l,l',p,p' }$ gives
\eqnb\label{cor_applied_to_w_critical_sowhat}
\| w \|_{p' , p}\leq \frac{1}{2} \| w \|_{p',p} + C \left( \| v \|_{L^{q',q}(Q)}^2 + \| v \|_{L^{q',q}(Q)} \right)
\eqne
Thus, subtracting the first term on the right hand side we obtain $w\in L^{p',p}(Q)$, which gives in particular that $v\in L^{p',p}(\widetilde{Q})$, as required. 

Note that, similarly as in step 1 of the proof of Theorem \ref{thm_subcritical_case}, this subtraction requires a rigorous justification, and can be verified similarly as in step 1' of that proof.
\end{proof}
\subsection{The case $q\in (1,2]$}\label{sec_the_case_q_1_2}
Here we briefly show that if $q,q'$ satisfy $q\in (1,2]$ and $1/q+4/q'\leq 1$ then one can apply a similar argument as in \eqref{cor_applied_to_w_critical} to obtain that $v\in L^{p',p}(Q')$ for $p=2$, any $p'\in (\max(q',8/3), \infty )$ and any subcylinder $Q'\Subset Q$. Note that, since $1/p+4/p'<1$, the claim (i.e. $u_x \in C^\infty (Q)$) then follows by Theorem \ref{thm_subcritical_case}.\vspace{0.1cm}\\

In order to show that $v\in L^{p',p}(Q')$ observe that the last term of \eqref{repr_of_w_again} can be bounded in the $L^{p',p}$ norm by
\[
C\left( \| v^2 \|_{L^{8/3,1} (Q)} + \| v \|_{L^{q',q}(Q) } \right) ,
\]
where we used Corollary \ref{cor_est_nonhomo_BHE_in_cylinder} again (recall that $v\in L^{16/3,2}(Q)$ (see \eqref{reg_for_q_between_1_2}) and observe that the choice of exponents $r'=p'$, $r=2$, $l'=8/3$, $l=1$ satisfies \eqref{exponents_not_sharp} with $k=2$). Thus \eqref{cor_applied_to_w_critical} follows as in the proof above with $\| v^2 \|_{L^{q'/2,q/2}(Q)}$ replaced by $\| v^2 \|_{L^{8/3,1}(Q)}$. As in \eqref{cor_applied_to_w_critical_sowhat} we arrive at
\[
\| w \|_{p' , p}\leq \frac{1}{2} \| w \|_{p',p} + C \left( \| v \|_{L^{16/3,2}(Q)}^2 + \| v \|_{L^{q',q}(Q)} \right),
\]
from which we deduce that $v\in L^{p',p} (\widetilde{Q} ) $ for any $\widetilde{Q} \Subset Q$ (as in the proof above; namely by regularising the equation and taking a limit as in step 1' of the proof of Theorem \ref{thm_subcritical_case}). \vspace{0.1cm}\\

Finally, observe that the issue discussed in this section does not appear in the case of the Navier--Stokes equation. In fact, in the case of the NSE the dimension of space is larger that the order of the nonlinearity (i.e. $2$); to be more precise the constraint of the exponents $2/q'+3/q \leq 1$ (recall \eqref{serrin_cond_NSE}) implies $q \geq 3>2$.

\section*{Conclusion}
We have proved that a weak solution to the surface growth model \eqref{SGM} on a cylinder $Q$ is smooth if it satisfies the local Serrin condition $u\in L^{q',q} (Q)$ where $q\in [2,\infty ]$, $q'\in [4, \infty ]$ are such that either $1/q+4/q' <1$ or $1/q+4/q'=1$, $q'<\infty$.
Our analysis excludes the endpoint case $q=1$, $q'=\infty$ since in this case we cannot assume the smallness condition \eqref{smallness_condition_sharp_case}. This case is not only particularly interesting, but also particularly challenging. In fact, the analogous problem in the case of the NSE is the problem whether a weak solution $u\in L^{\infty , 3}(Q)$ to the NSE on $Q$ is regular in $Q$, and it was resolved in a deep paper by \cite{ESS_2003} using a blow-up technique. We believe that, despite the fact that the SGM is a one-dimensional equation, it might be a more difficult problem than in the case of the NSE, since a number of techniques used by \cite{ESS_2003} seem unavailable in the SGM (such as the $L_{s,l}$-coercive estimates for solutions of the non-stationary Stokes system) and since the $L^1$ space is not reflexive.

However, given the number of similarities between the SGM and the NSE, as developed in the recent years, it is expected that the case $q=1$, $q'=\infty$ gives  regularity as well.
\section*{Acknowledgements}

I would like to thank James Robinson for the careful reading of a draft of this article. His numerous comments significantly improved its quality.  

The author is supported by EPSRC as part of the MASDOC DTC at the University of Warwick, Grant No. EP/HO23364/1.
\bibliography{literature}{}

\end{document}